\documentclass[preprint,number]{elsarticle}

\usepackage{amsthm,amssymb,amsmath,tikz, enumitem, circuitikz,graphicx}
\usetikzlibrary{decorations.pathreplacing,calligraphy,patterns}
\newtheorem{thm}{Theorem}[section]

\newtheorem{lmm}[thm]{Lemma}

\begin{document}

\begin{abstract}
    The family of cycle completable graphs has several cryptomorphic descriptions, the equivalence of which has heretofore been proven by a laborious implication-cycle that detours through a motivating matrix completion problem. We give a concise proof, partially by introducing a new characterization. Then we generalize this family to ``$k$-quasichordal'' graphs, with three natural characterizations.
    \end{abstract}

\begin{frontmatter}
\title{Characterizing and generalizing 
cycle completable graphs}
\author[0]{Maria Chudnovsky}
\author[0]{Ian M.J. McInnis\corref{cor1}}

\affiliation[0]{organization={Department of Mathematics, Princeton University},
                addressline
                city={Princeton},
                postcode={08544}, 
                state={NJ},
                country={USA}}
\cortext[cor1]{Corresponding author. imj@alumni.princeton.edu}

\end{frontmatter}

\section{Introduction}\label{Introduction}

A class of graphs called \textit{cycle completable} graphs has arisen in various matrix completion problems \cite{BJL}, especially those that pertain to geometry \cite{JJK, McI}. 
There are several cryptomorphic characterizations; Theorem \ref{originalequivalence}, which we prove in Section \ref{Characterizing}, states their equivalence. We preserve their numbering in \cite{JM} but adapt their wording. Condition $(\alpha)$ is new. See Section \ref{Definitions} for definitions.

\begin{thm}\label{originalequivalence}
    The following are equivalent:
\end{thm}
\begin{enumerate}
    \item[(1)] No induced subgraph of $G$ is equal to or built from $\widehat{W}_4$ or $W_n$ for $n\geq 5$.
    \item[\text{(2)}] Every induced subgraph of $G$ is series-parallel or contains a $K_4$ subgraph.
    \item[\text{($\alpha$)}] $G$ is $\mathcal{F}$-free.
    \item[\text{(A)}] $G$ is a clique-sum of chordal and series-parallel graphs.
    \item[\text{(B)}] $G$ has a mixed elimination ordering.
    \item[\text{(3)}] $G$ has a chordal supergraph $H$ such that any $K_4\subseteq H$ is also a $K_4\subseteq G$. 
\end{enumerate}

Characterizations (1), (2), and (3) arose in connection with positive definite matrices in \cite{BJL}. That paper stated a matrix-theoretic condition (0) and proved $(0)\rightarrow (1)\rightarrow(2)\rightarrow(3) \rightarrow (0)$. All but a page of the implication $(2)\rightarrow(3)$ was devoted to a technical lemma that \cite{JM} used to prove $(2)\rightarrow (A)$, appending thereunto a two-paragraph proof that $(A)\rightarrow (B)\rightarrow (3)$. Both \cite{JJK} and \cite{McI} prove $(3)\rightarrow (1)$ via conditions that pertain more to metric geometry than to linear operators or quadratic forms. Besides being long, these proofs are still thoroughly matricial/numerical. The implications $(1)\rightarrow (2) \rightarrow (A) \rightarrow (B)\rightarrow (3)$ were proven by thirty-four pages of laborious (albeit graph-theoretic) casework. We seek a short, matrix-free proof that $(1),(2),(A),(B),$ and $(3)$ are equivalent. 

To that end, we introduce $(\alpha)$. The implications $(1)\rightarrow (\alpha)$ and $(A)\rightarrow (\alpha)$ are easier than $(1)\rightarrow (2)$ and $(A)\rightarrow (2)$, and $(\alpha)\rightarrow (2)$ follows from a short lemma in \cite{LMT}. Our $(2)\rightarrow (A)$ does not use $(\alpha)$ but is much shorter than before.
The original $(A)\rightarrow (B) \rightarrow (3)$ is effective; we offer no replacement. Finally, $(3)\rightarrow (A)$ and $(2)\rightarrow (1)$ are easy. Our proof could be shortened further; we have occasionally sacrificed some brevity in order to highlight interesting, supererogatory lemmata.\\

Other matrix completion problems yield a natural generalization of $(3)$ in which $K_4$ is replaced by $K_n$, where $n$ varies with the problem.\footnote{The problems pertain to families of metric spaces. The $n$ in $K_n$ pertains to a parameter of the family called the ``dullness.'' See \cite{McI} for details.}
With a little definition-craft, we can find equivalents of $(A)$ and $(B)$. The graphs thus described we call \emph{$k$-quasichordal}. Chordal graphs are the cases $k=1$; cycle completable graphs are the case $k=2$. Each characterization captures the notion that ``locally, $G$ is clique-like or partial-$k$-tree-like.'' See Section \ref{Generalizing} for the proof.
\begin{thm}\label{TheKversion}
    The following are equivalent for $k\geq 1$:
\end{thm}
    \begin{enumerate}
        \item[\textbf{$(A_k)$}] $G$ is a  clique-sum of cliques and graphs having treewidth $\leq k$.
        \item[\textbf{$(B_k)$}] $G$ has a $k$-blended elimination ordering.
        \item[\textbf{$(3_k)$}] $G$ has a $(k+1)$-clique chordal supergraph. 
\end{enumerate}
Proving $(3_k)\rightarrow (A_k)$ is no harder than $(3)\rightarrow (A)$; Lemma \ref{3kimpliesAk} entails it. 
The original $(A)\rightarrow (B) \rightarrow (3)$ generalizes succinctly but dully, so for the sake of novelty and nicer lemmata, we prove $(A_k)\rightarrow (3_k)$ and $(B_k)\leftrightarrow (3_k)$. 

Generalizing $(1),(2)$, or $(\alpha)$ would consist of giving an induced subgraph obstruction set for $k$-quasichordal graphs. This would be quite useful, since forbidden induced subgraphs arise naturally in the proof structure of the motivating matricial problems \cite{BJL, JJK, McI}. Unfortunately, an obstruction set for $k$-quasichordal graphs is contained in (and includes most members of) an obstruction set for graphs of treewidth $\leq k$, which have proven quite difficult to find. See \cite{ACV} and its sequels. 

\section{Definitions}\label{Definitions}
Throughout, graphs are undirected and simple. We ignore the distinction between isomorphism of graphs and their equality.

The complete graph on $p$ vertices is denoted $K_p$. The cycle on $p$ vertices is $C_p$. The complete bipartite graph with $p$ vertices on one side and $q$ on the other is $K_{p,q}$. If vertices $u,v$ are adjacent, we write $u \sim v$; otherwise we write $u\not\sim v$. The vertices of G are $V(G)$; its edges are $E(G)$. Denote by $N(v)$ the neighbors of v. 
If $H$ can be obtained from $G$ by deleting vertices, we call $H$ an \emph{induced subgraph} of $G$ and write $H=G|_{V(H)}$. A $p$-\textit{clique} is an induced $K_p$ subgraph. If $H$ can be obtained from $G$ by deleting vertices, deleting edges, and contracting edges, we call $H$ a \textit{minor} of $G$. (Recall that a graph is series-parallel if and only if it has no $K_4$ minor.)

A \emph{tree} is a connected, nonempty graph with no cycles.
For a graph $G$, a \emph{tree decomposition} is a pair $(T,f)$ consisting of a tree $T$ and a function $f: V(G) \rightarrow 2^{V(T)}$, satisfying the following properties:
\begin{itemize}
    \item For any $x\in V(G)$, $T|_{f(x)}$ is a tree.
    \item If $xy\in E(G)$, $f(x)\cap f(y)\neq \emptyset$.
\end{itemize}
For a vertex $v\in V(T)$, $f^{-1}(v)$ is the \emph{bag} associated to $v$. The \emph{width} of a tree decomposition $(T,f)$ is equal to $\max_{v\in V(T)}|f^{-1}(v)|-1$. The \emph{treewidth} of a graph $G$ is the minimum width of a tree decomposition of $G$. If $G$ has treewidth $\leq k$, we say that it is a \emph{partial $k$-tree.} (Treewidth's naturality and importance have been independently discovered in dynamic programming \cite{BB}, graph functions \cite{Hal}, and graph minor theory \cite{RS}.)

A \emph{prism} consists of two triangles on distinct vertices $u_1,u_2,u_3,v_1,v_2,v_3$, plus three pairwise disjoint paths: one from $u_i$ to $v_i$ for each $i\in \{1,2,3\}$. 
A \emph{wheel} consists of a $C_k$ for some $k\geq 4$ and an additional vertex $v$ not in $C_k$ adjacent to at least three vertices of the $C_k$. We denote the minimal wheel by $\widehat{W}_4$. If $v$ is adjacent to each of the $k$ other vertices, we call it a \emph{universal wheel} and denote it $W_n$, where $n=k+1$. 

A \emph{vertex partition} of $G$ is a graph $G'$ in which a vertex $v$ of $G$ is replaced by two adjacent vertices $x,y$, such that $(N(x), N(y))$ partitions $N(v)$. 
We say that $G_2$ is \emph{built from} $G_1$ just in case $G_2$ is obtained from $G_1$ by a (finite) sequence of vertex partitions.\footnote{In earlier references \cite{BJL, JJK, JM} the sequence was required to be nonempty. We find this awkward but will use the phrase ``equal to or built from'' for backwards compatibility.} 

The family $\mathcal{F}$ includes $K_{3,3}$, wheels, prisms, and graphs equal to or built from $\widehat{W}_4$. If no induced subgraph $H\subseteq G$ is in $\mathcal{F}$, we say that $G$ is \emph{$\mathcal{F}$-free.} The diagram below depicts members of $\mathcal{F}$. The dashed lines indicate paths of arbitrary length (including single edges), and the dotted lines indicate single edges that may be present or not.

\begin{figure}[!htbp]
	\centering
	\resizebox{.75\textwidth}{!}{%
	\begin{circuitikz}
	\tikzstyle{every node}=[font=\LARGE]
	
	\draw  (2.25,9.5) circle (0cm);
	\draw (0,11) to[short, -*] (0,11);
	\draw (-2.5,11) to[short, -*] (-2.5,11);
	\draw (-2.5,12.25) to[short, -*] (-2.5,12.25);
	\draw (-2.5,9.75) to[short, -*] (-2.5,9.75);
	\draw (0,12.25) to[short, -*] (0,12.25);
	\draw (0,9.75) to[short, -*] (0,9.75);
	\draw [short] (-2.5,9.75) -- (0,11);
	\draw [short] (-2.5,9.75) -- (0,9.75);
	\draw [short] (-2.5,9.75) -- (0,12.25);
	\draw [short] (0,9.75) -- (-2.5,11);
	\draw [short] (0,9.75) -- (-2.5,12.25);
	\draw [short] (0,11) -- (-2.5,12.25);
	\draw [short] (0,11) -- (-2.5,11);
	\draw [short] (0,12.25) -- (-2.5,11);
	\draw [short] (-2.5,12.25) -- (0,12.25);
	\draw (3.75,12.25) to[short, -*] (3.75,12.25);
	\draw (3.75,9.75) to[short, -*] (3.75,9.75);
	\draw (5.75,11) to[short, -*] (5.75,11);
	\draw (11.25,9.75) to[short, -*] (11.25,9.75);
	\draw (9.25,11) to[short, -*] (9.25,11);
	\draw (11.25,12.25) to[short, -*] (11.25,12.25);
	\draw [short] (11.25,9.75) -- (11.25,12.25);
	\draw [short] (11.25,12.25) -- (9.25,11);
	\draw [short] (9.25,11) -- (11.25,9.75);
	\draw [short] (3.75,9.75) -- (3.75,12.25);
	\draw [short] (3.75,12.25) -- (5.75,11);
	\draw [short] (5.75,11) -- (3.75,9.75);
	\draw [dashed] (3.75,9.75) -- (11.25,9.75);
	\draw [dashed] (9.25,11) -- (5.75,11);
	\draw [dashed] (3.75,12.25) -- (11.25,12.25);
	\draw (-2.5,6) to[short, -*] (-2.5,6);
	\draw (0,3.5) to[short, -*] (0,3.5);
	\draw (11.25,6) to[short, -*] (11.25,6);
	\draw (8.75,8.5) to[short, -*] (8.75,8.5);
	\draw (6.25,6) to[short, -*] (6.25,6);
	\draw (8.75,3.5) to[short, -*] (8.75,3.5);
	\draw (0,6) to[short, -*] (0,6);
	\draw (0,8.5) to[short, -*] (0,8.5);
	\draw [dashed] (6.25,6) -- (8.75,8.5);
	\draw [dashed] (8.75,8.5) -- (11.25,6);
	\draw [dashed] (11.25,6) -- (8.75,3.5);
	\draw [dashed] (8.75,3.5) -- (6.25,6);
	\draw (8.75,6) to[short, -*] (8.75,6);
	\draw [dashed] (0,6) -- (0,8.5);
	\draw [dashed] (0,8.5) -- (-2.5,6);
	\draw [dashed] (-2.5,6) -- (-0.25,6);
	\draw [dashed] (0,6) -- (2.5,6);
	\draw [dashed] (2.5,6) -- (0,8.5);
	\draw (2.5,6) to[short, -*] (2.5,6);
	\draw [dashed] (-2.5,6) -- (0,3.5);
	\draw [dashed] (2.5,6) -- (0,3.5);
	\draw [short] (6.25,6) -- (8.75,6);
	\draw [short] (8.75,6) -- (8.75,8.5);
	\draw [short] (8.75,6) -- (11.25,6);
	\draw [dotted] (8.75,6) -- (8.75,3.5);
	\draw [dotted] (8.75,6) -- (9.5,4.25);
	\draw [dotted] (8.75,6) -- (9.75,4.75);
	\draw [dotted] (8.75,6) -- (10.25,5.25);
	\draw [dotted] (8.75,6) -- (8.25,4.25);
	\draw [dotted] (8.75,6) -- (7.75,4.75);
	\draw [dotted] (8.75,6) -- (7.5,5.75);
	\draw [dotted] (8.75,6) -- (10.5,6.5);
	\draw [dotted] (9,6) -- (9.5,7);
	\draw [dotted] (8.75,6) -- (9.25,7.5);
	\draw [dotted] (8.75,6) -- (8.25,7.5);
	\draw [dotted] (8.75,6) -- (8,7);
	\draw [dotted] (8.75,6) -- (7.25,6.75);
	\end{circuitikz}
	}%
	
	\label{fig:my_label}
	\end{figure}

Given two graphs $G_1$ and $G_2$ and cliques $K_p\subseteq G_1$, $K_p\subseteq G_2$, a graph $G$ obtained by bijectively identifying the vertices of these cliques is called a \emph{clique-sum} of $G_1$ and $G_2$. If a graph $G'$ can be formed by clique-sums among $G_1,...,G_n$, we say that $G'$ is a clique-sum of $G_1,...,G_n$. (Some authors say ``sequential clique-sum.'') A list of such $G_i$ is said to be a \emph{cut-clique decomposition} of $G'$; a \emph{cut-clique} of a graph $G$ is a clique $K_p\subseteq G$ such that $G\backslash K_p$ has more components than $G$. 
  
We say that $G'$ is a \emph{$k$-clique chordal supergraph} of $G$ if $V(G)=V(G')$, $E(G)\subseteq E(G')$, $G'$ is chordal, and $G'$ contains no $(k+1)$-cliques not contained by $G$. (Recall that the treewidth of a graph $G$ is equal to $\min \omega(G')-1$, $\omega$ denotes the size of the largest clique, and where the minimum is taken over all chordal supergraphs $G'\supseteq G$.)

For a graph $G$, let $\pi=(v_1,...,v_{|V(G)|})$ be an ordering of the vertices of $G$. 
Denote by $G_{\pi}$ the \emph{fill-in graph of $G$ with respect to $\pi$}. It is constructed as follows. Start with $G$. For each $i\in \{1,...,|V(G)|\}$, consider the vertex $v_i$. For every $j,k\in \{i+1,...,|V(G)|\}$, if $j\neq k$ and $v_j\sim v_i \sim v_k \not \sim v_j$, add an edge between $v_j$ and $v_k$. Let $N^+_\pi(v_i)$ denote the higher-numbered neighbors of $v_i$ in $G_\pi$.  
Fix an integer $p$, and let $\pi$ be a vertex ordering such that for every $v_i$, either 
\begin{enumerate}
    \item[(i)] $N^+_\pi(v_i)$ is complete in $G$, or  
    \item[(ii)] $|N^+_\pi(v_i)|\leq p$. 
\end{enumerate}
Then we call $\pi$ a \emph{$p$-blended elimination ordering.}

\section{Characterizing}\label{Characterizing}
\begin{lmm}\label{Fiswheelfounded}
    Prisms and $K_{3,3}$ are built from $W_{5}$. Wheels are equal to or built from $\widehat{W}_4$ or $W_k$ for $k\geq 5$. 
\end{lmm}

\begin{proof}
    $K_{3,3}$ is built from $W_5$: partition the hub so that diagonally opposite edges lie on the same side of the partition. 
    Any prism is equal to or built from the six-vertex prism, which is built from $W_5$: partition the hub so that diagonally opposite edges lie on different sides of the partition.


\tikzset{
    pattern size/.store in=\mcSize, 
    pattern size = 5pt,
    pattern thickness/.store in=\mcThickness, 
    pattern thickness = 0.3pt,
    pattern radius/.store in=\mcRadius, 
    pattern radius = 1pt}
    \makeatletter
    \pgfutil@ifundefined{pgf@pattern@name@_4eakzwuf2}{
    \pgfdeclarepatternformonly[\mcThickness,\mcSize]{_4eakzwuf2}
    {\pgfqpoint{-\mcThickness}{-\mcThickness}}
    {\pgfpoint{\mcSize}{\mcSize}}
    {\pgfpoint{\mcSize}{\mcSize}}
    {
    \pgfsetcolor{\tikz@pattern@color}
    \pgfsetlinewidth{\mcThickness}
    \pgfpathmoveto{\pgfpointorigin}
    \pgfpathlineto{\pgfpoint{0}{\mcSize}}
    \pgfusepath{stroke}
    }}
    \makeatother
    
     
    \tikzset{
    pattern size/.store in=\mcSize, 
    pattern size = 5pt,
    pattern thickness/.store in=\mcThickness, 
    pattern thickness = 0.3pt,
    pattern radius/.store in=\mcRadius, 
    pattern radius = 1pt}
    \makeatletter
    \pgfutil@ifundefined{pgf@pattern@name@_ygqco2urn}{
    \pgfdeclarepatternformonly[\mcThickness,\mcSize]{_ygqco2urn}
    {\pgfqpoint{-\mcThickness}{-\mcThickness}}
    {\pgfpoint{\mcSize}{\mcSize}}
    {\pgfpoint{\mcSize}{\mcSize}}
    {
    \pgfsetcolor{\tikz@pattern@color}
    \pgfsetlinewidth{\mcThickness}
    \pgfpathmoveto{\pgfpointorigin}
    \pgfpathlineto{\pgfpoint{0}{\mcSize}}
    \pgfusepath{stroke}
    }}
    \makeatother
    
     
    \tikzset{
    pattern size/.store in=\mcSize, 
    pattern size = 5pt,
    pattern thickness/.store in=\mcThickness, 
    pattern thickness = 0.3pt,
    pattern radius/.store in=\mcRadius, 
    pattern radius = 1pt}
    \makeatletter
    \pgfutil@ifundefined{pgf@pattern@name@_rfyr3e2ks}{
    \pgfdeclarepatternformonly[\mcThickness,\mcSize]{_rfyr3e2ks}
    {\pgfqpoint{-\mcThickness}{-\mcThickness}}
    {\pgfpoint{\mcSize}{\mcSize}}
    {\pgfpoint{\mcSize}{\mcSize}}
    {
    \pgfsetcolor{\tikz@pattern@color}
    \pgfsetlinewidth{\mcThickness}
    \pgfpathmoveto{\pgfpointorigin}
    \pgfpathlineto{\pgfpoint{0}{\mcSize}}
    \pgfusepath{stroke}
    }}
    \makeatother
    \tikzset{every picture/.style={line width=0.5pt}} 
     
\begin{figure}[!htbp]
	\centering
    \begin{tikzpicture}[x=0.75pt,y=0.75pt,yscale=-1,xscale=1]
    
    \draw    (79.4,99.8) -- (79.4,140.2) ;
    \draw [shift={(79.4,140.2)}, rotate = 90] [color={rgb, 255:red, 0; green, 0; blue, 0 }  ][fill={rgb, 255:red, 0; green, 0; blue, 0 }  ][line width=0.75]      (0, 0) circle [x radius= 1.34, y radius= 1.34]   ;
    \draw [shift={(79.4,99.8)}, rotate = 90] [color={rgb, 255:red, 0; green, 0; blue, 0 }  ][fill={rgb, 255:red, 0; green, 0; blue, 0 }  ][line width=0.75]      (0, 0) circle [x radius= 1.34, y radius= 1.34]   ;
    \draw [pattern=_4eakzwuf2,pattern size=6pt,pattern thickness=0.75pt,pattern radius=0pt, pattern color={rgb, 255:red, 0; green, 0; blue, 0}] [dash pattern={on 4.5pt off 4.5pt}]  (79.4,99.8) -- (158.6,99.8) ;
    \draw    (158.6,99.8) -- (158.6,140.2) ;
    \draw [shift={(158.6,140.2)}, rotate = 90] [color={rgb, 255:red, 0; green, 0; blue, 0 }  ][fill={rgb, 255:red, 0; green, 0; blue, 0 }  ][line width=0.75]      (0, 0) circle [x radius= 1.34, y radius= 1.34]   ;
    \draw [shift={(158.6,99.8)}, rotate = 90] [color={rgb, 255:red, 0; green, 0; blue, 0 }  ][fill={rgb, 255:red, 0; green, 0; blue, 0 }  ][line width=0.75]      (0, 0) circle [x radius= 1.34, y radius= 1.34]   ;
    \draw    (240.4,100.2) -- (240.4,140.6) ;
    \draw [shift={(240.4,140.6)}, rotate = 90] [color={rgb, 255:red, 0; green, 0; blue, 0 }  ][fill={rgb, 255:red, 0; green, 0; blue, 0 }  ][line width=0.75]      (0, 0) circle [x radius= 1.34, y radius= 1.34]   ;
    \draw [shift={(240.4,100.2)}, rotate = 90] [color={rgb, 255:red, 0; green, 0; blue, 0 }  ][fill={rgb, 255:red, 0; green, 0; blue, 0 }  ][line width=0.75]      (0, 0) circle [x radius= 1.34, y radius= 1.34]   ;
    \draw    (200,100.2) -- (200,140.6) ;
    \draw [shift={(200,140.6)}, rotate = 90] [color={rgb, 255:red, 0; green, 0; blue, 0 }  ][fill={rgb, 255:red, 0; green, 0; blue, 0 }  ][line width=0.75]      (0, 0) circle [x radius= 1.34, y radius= 1.34]   ;
    \draw [shift={(200,100.2)}, rotate = 90] [color={rgb, 255:red, 0; green, 0; blue, 0 }  ][fill={rgb, 255:red, 0; green, 0; blue, 0 }  ][line width=0.75]      (0, 0) circle [x radius= 1.34, y radius= 1.34]   ;
    \draw [pattern=_ygqco2urn,pattern size=6pt,pattern thickness=0.75pt,pattern radius=0pt, pattern color={rgb, 255:red, 0; green, 0; blue, 0}] [dash pattern={on 4.5pt off 4.5pt}]  (79.4,140.2) -- (158.6,140.2) ;
    \draw    (79.4,99.8) -- (99.4,119.8) ;
    \draw [shift={(99.4,119.8)}, rotate = 45] [color={rgb, 255:red, 0; green, 0; blue, 0 }  ][fill={rgb, 255:red, 0; green, 0; blue, 0 }  ][line width=0.75]      (0, 0) circle [x radius= 1.34, y radius= 1.34]   ;
    \draw    (99.4,119.8) -- (79.4,140.2) ;
    \draw [shift={(79.4,140.2)}, rotate = 134.43] [color={rgb, 255:red, 0; green, 0; blue, 0 }  ][fill={rgb, 255:red, 0; green, 0; blue, 0 }  ][line width=0.75]      (0, 0) circle [x radius= 1.34, y radius= 1.34]   ;
    \draw    (158.6,99.8) -- (139,119.8) ;
    \draw [shift={(139,119.8)}, rotate = 134.42] [color={rgb, 255:red, 0; green, 0; blue, 0 }  ][fill={rgb, 255:red, 0; green, 0; blue, 0 }  ][line width=0.75]      (0, 0) circle [x radius= 1.34, y radius= 1.34]   ;
    \draw    (139,119.8) -- (158.6,140.2) ;
    \draw [shift={(158.6,140.2)}, rotate = 46.15] [color={rgb, 255:red, 0; green, 0; blue, 0 }  ][fill={rgb, 255:red, 0; green, 0; blue, 0 }  ][line width=0.75]      (0, 0) circle [x radius= 1.34, y radius= 1.34]   ;
    \draw [pattern=_rfyr3e2ks,pattern size=6pt,pattern thickness=0.75pt,pattern radius=0pt, pattern color={rgb, 255:red, 0; green, 0; blue, 0}] [dash pattern={on 4.5pt off 4.5pt}]  (99.4,119.8) -- (139,119.8) ;
    \draw    (240.4,140.6) -- (200,140.6) ;
    \draw [shift={(200,140.6)}, rotate = 180] [color={rgb, 255:red, 0; green, 0; blue, 0 }  ][fill={rgb, 255:red, 0; green, 0; blue, 0 }  ][line width=0.75]      (0, 0) circle [x radius= 1.34, y radius= 1.34]   ;
    \draw [shift={(240.4,140.6)}, rotate = 180] [color={rgb, 255:red, 0; green, 0; blue, 0 }  ][fill={rgb, 255:red, 0; green, 0; blue, 0 }  ][line width=0.75]      (0, 0) circle [x radius= 1.34, y radius= 1.34]   ;
    \draw    (240.4,100.2) -- (200,100.2) ;
    \draw [shift={(200,100.2)}, rotate = 180] [color={rgb, 255:red, 0; green, 0; blue, 0 }  ][fill={rgb, 255:red, 0; green, 0; blue, 0 }  ][line width=0.75]      (0, 0) circle [x radius= 1.34, y radius= 1.34]   ;
    \draw [shift={(240.4,100.2)}, rotate = 180] [color={rgb, 255:red, 0; green, 0; blue, 0 }  ][fill={rgb, 255:red, 0; green, 0; blue, 0 }  ][line width=0.75]      (0, 0) circle [x radius= 1.34, y radius= 1.34]   ;
    \draw    (319.4,100.6) -- (319.4,141) ;
    \draw [shift={(319.4,141)}, rotate = 90] [color={rgb, 255:red, 0; green, 0; blue, 0 }  ][fill={rgb, 255:red, 0; green, 0; blue, 0 }  ][line width=0.75]      (0, 0) circle [x radius= 1.34, y radius= 1.34]   ;
    \draw [shift={(319.4,100.6)}, rotate = 90] [color={rgb, 255:red, 0; green, 0; blue, 0 }  ][fill={rgb, 255:red, 0; green, 0; blue, 0 }  ][line width=0.75]      (0, 0) circle [x radius= 1.34, y radius= 1.34]   ;
    \draw    (279,100.6) -- (279,141) ;
    \draw [shift={(279,141)}, rotate = 90] [color={rgb, 255:red, 0; green, 0; blue, 0 }  ][fill={rgb, 255:red, 0; green, 0; blue, 0 }  ][line width=0.75]      (0, 0) circle [x radius= 1.34, y radius= 1.34]   ;
    \draw [shift={(279,100.6)}, rotate = 90] [color={rgb, 255:red, 0; green, 0; blue, 0 }  ][fill={rgb, 255:red, 0; green, 0; blue, 0 }  ][line width=0.75]      (0, 0) circle [x radius= 1.34, y radius= 1.34]   ;
    \draw    (319.4,141) -- (279,141) ;
    \draw [shift={(279,141)}, rotate = 180] [color={rgb, 255:red, 0; green, 0; blue, 0 }  ][fill={rgb, 255:red, 0; green, 0; blue, 0 }  ][line width=0.75]      (0, 0) circle [x radius= 1.34, y radius= 1.34]   ;
    \draw [shift={(319.4,141)}, rotate = 180] [color={rgb, 255:red, 0; green, 0; blue, 0 }  ][fill={rgb, 255:red, 0; green, 0; blue, 0 }  ][line width=0.75]      (0, 0) circle [x radius= 1.34, y radius= 1.34]   ;
    \draw    (319.4,100.6) -- (279,100.6) ;
    \draw [shift={(279,100.6)}, rotate = 180] [color={rgb, 255:red, 0; green, 0; blue, 0 }  ][fill={rgb, 255:red, 0; green, 0; blue, 0 }  ][line width=0.75]      (0, 0) circle [x radius= 1.34, y radius= 1.34]   ;
    \draw [shift={(319.4,100.6)}, rotate = 180] [color={rgb, 255:red, 0; green, 0; blue, 0 }  ][fill={rgb, 255:red, 0; green, 0; blue, 0 }  ][line width=0.75]      (0, 0) circle [x radius= 1.34, y radius= 1.34]   ;
    \draw    (240.4,100.2) -- (220,120.2) ;
    \draw [shift={(220,120.2)}, rotate = 135.57] [color={rgb, 255:red, 0; green, 0; blue, 0 }  ][fill={rgb, 255:red, 0; green, 0; blue, 0 }  ][line width=0.75]      (0, 0) circle [x radius= 1.34, y radius= 1.34]   ;
    \draw [shift={(240.4,100.2)}, rotate = 135.57] [color={rgb, 255:red, 0; green, 0; blue, 0 }  ][fill={rgb, 255:red, 0; green, 0; blue, 0 }  ][line width=0.75]      (0, 0) circle [x radius= 1.34, y radius= 1.34]   ;
    \draw    (240.4,140.6) -- (220,120.2) ;
    \draw [shift={(220,120.2)}, rotate = 225] [color={rgb, 255:red, 0; green, 0; blue, 0 }  ][fill={rgb, 255:red, 0; green, 0; blue, 0 }  ][line width=0.75]      (0, 0) circle [x radius= 1.34, y radius= 1.34]   ;
    \draw [shift={(240.4,140.6)}, rotate = 225] [color={rgb, 255:red, 0; green, 0; blue, 0 }  ][fill={rgb, 255:red, 0; green, 0; blue, 0 }  ][line width=0.75]      (0, 0) circle [x radius= 1.34, y radius= 1.34]   ;
    \draw    (200,140.6) -- (220,120.2) ;
    \draw [shift={(220,120.2)}, rotate = 314.43] [color={rgb, 255:red, 0; green, 0; blue, 0 }  ][fill={rgb, 255:red, 0; green, 0; blue, 0 }  ][line width=0.75]      (0, 0) circle [x radius= 1.34, y radius= 1.34]   ;
    \draw [shift={(200,140.6)}, rotate = 314.43] [color={rgb, 255:red, 0; green, 0; blue, 0 }  ][fill={rgb, 255:red, 0; green, 0; blue, 0 }  ][line width=0.75]      (0, 0) circle [x radius= 1.34, y radius= 1.34]   ;
    \draw    (200,100.2) -- (220,120.2) ;
    \draw [shift={(220,120.2)}, rotate = 45] [color={rgb, 255:red, 0; green, 0; blue, 0 }  ][fill={rgb, 255:red, 0; green, 0; blue, 0 }  ][line width=0.75]      (0, 0) circle [x radius= 1.34, y radius= 1.34]   ;
    \draw [shift={(200,100.2)}, rotate = 45] [color={rgb, 255:red, 0; green, 0; blue, 0 }  ][fill={rgb, 255:red, 0; green, 0; blue, 0 }  ][line width=0.75]      (0, 0) circle [x radius= 1.34, y radius= 1.34]   ;
    \draw    (319.4,100.6) -- (309,120.6) ;
    \draw [shift={(309,120.6)}, rotate = 117.47] [color={rgb, 255:red, 0; green, 0; blue, 0 }  ][fill={rgb, 255:red, 0; green, 0; blue, 0 }  ][line width=0.75]      (0, 0) circle [x radius= 1.34, y radius= 1.34]   ;
    \draw [shift={(319.4,100.6)}, rotate = 117.47] [color={rgb, 255:red, 0; green, 0; blue, 0 }  ][fill={rgb, 255:red, 0; green, 0; blue, 0 }  ][line width=0.75]      (0, 0) circle [x radius= 1.34, y radius= 1.34]   ;
    \draw    (279,100.6) -- (288.8,120.6) ;
    \draw [shift={(288.8,120.6)}, rotate = 63.9] [color={rgb, 255:red, 0; green, 0; blue, 0 }  ][fill={rgb, 255:red, 0; green, 0; blue, 0 }  ][line width=0.75]      (0, 0) circle [x radius= 1.34, y radius= 1.34]   ;
    \draw [shift={(279,100.6)}, rotate = 63.9] [color={rgb, 255:red, 0; green, 0; blue, 0 }  ][fill={rgb, 255:red, 0; green, 0; blue, 0 }  ][line width=0.75]      (0, 0) circle [x radius= 1.34, y radius= 1.34]   ;
    \draw    (309,120.6) -- (288.8,120.6) ;
    \draw [shift={(288.8,120.6)}, rotate = 180] [color={rgb, 255:red, 0; green, 0; blue, 0 }  ][fill={rgb, 255:red, 0; green, 0; blue, 0 }  ][line width=0.75]      (0, 0) circle [x radius= 1.34, y radius= 1.34]   ;
    \draw [shift={(309,120.6)}, rotate = 180] [color={rgb, 255:red, 0; green, 0; blue, 0 }  ][fill={rgb, 255:red, 0; green, 0; blue, 0 }  ][line width=0.75]      (0, 0) circle [x radius= 1.34, y radius= 1.34]   ;
    \draw    (319.4,141) -- (288.8,120.6) ;
    \draw [shift={(288.8,120.6)}, rotate = 213.69] [color={rgb, 255:red, 0; green, 0; blue, 0 }  ][fill={rgb, 255:red, 0; green, 0; blue, 0 }  ][line width=0.75]      (0, 0) circle [x radius= 1.34, y radius= 1.34]   ;
    \draw [shift={(319.4,141)}, rotate = 213.69] [color={rgb, 255:red, 0; green, 0; blue, 0 }  ][fill={rgb, 255:red, 0; green, 0; blue, 0 }  ][line width=0.75]      (0, 0) circle [x radius= 1.34, y radius= 1.34]   ;
    \draw    (309,120.6) -- (279,141) ;
    \draw [shift={(279,141)}, rotate = 145.78] [color={rgb, 255:red, 0; green, 0; blue, 0 }  ][fill={rgb, 255:red, 0; green, 0; blue, 0 }  ][line width=0.75]      (0, 0) circle [x radius= 1.34, y radius= 1.34]   ;
    \draw [shift={(309,120.6)}, rotate = 145.78] [color={rgb, 255:red, 0; green, 0; blue, 0 }  ][fill={rgb, 255:red, 0; green, 0; blue, 0 }  ][line width=0.75]      (0, 0) circle [x radius= 1.34, y radius= 1.34]   ;
    \draw   (247.25,116.71) -- (260.57,116.71) -- (260.57,113.87) -- (269.45,119.54) -- (260.57,125.2) -- (260.57,122.37) -- (247.25,122.37) -- cycle ;
    \draw   (189.91,122.49) -- (176.6,122.82) -- (176.67,125.65) -- (167.65,120.21) -- (176.39,114.33) -- (176.46,117.16) -- (189.77,116.83) -- cycle ;
    \end{tikzpicture}
\end{figure}

    A wheel whose hub has degree $k\geq 4$ is built from $W_k$ by subdividing edges along the rim. A wheel whose hub has degree $3$ has a rim of at least four vertices (by definition) and so can be built from $\widehat{W}_{4}$ similarly. 
\end{proof}

\begin{lmm}\label{Onlyhardlemma}
    For a graph $G$, at least one of the following holds: 
    \begin{itemize}
        \item $G$ is series-parallel.
        \item $G$ is a clique.
        \item $G$ has a cut-clique.
        \item $G$ contains an induced subgraph $H$ that is not series-parallel nor contains a $K_4$ subgraph. 
    \end{itemize}
    \end{lmm} 

    \begin{proof}
        Let $G$ be a graph. Assume that $G$ is not a clique, is not series-parallel, and has no cut-clique. We seek to show that $G$ contains an induced subgraph $H$ where $H$ is not series-parallel but does not have a $K_4$ subgraph.
        
        Let $K$ be a maximum clique in $G$. Every vertex of $K$ must have a neighbor in $G\backslash K$; if some $v\in V(K)$ did not, then $V(K)\backslash v$ would be a clique cutset for $G$. 
    
        \emph{Case 1.} Assume $|V(K)|\leq 2$. Then $G$ contains no $K_4$ subgraph, but since it is not series-parallel, it has a $K_4$ minor, so we can take $H=G$. 
    
        \emph{Case 2.} Assume $|V(K)|\geq 3$ and the existence of $v_1,v_2\in V(K)$ such that $N(v_1)\cap N(v_2)\cap V(G\backslash K)=\emptyset$. Let $v_3$ be an arbitrary third vertex in $K$, and let $u_1,u_2,u_3\in V(G\backslash K)$ be neighbors of $v_1,v_2,v_3$ respectively, chosen such that the path distance from $u_1$ to $u_2$ is minimized.

\begin{figure}[!htbp]
	\centering
    \begin{tikzpicture}[x=0.75pt,y=0.75pt,yscale=-1,xscale=1]

        
        \draw    (100.25,179.87) -- (99.85,230.27) ;
        \draw [shift={(99.85,230.27)}, rotate = 90.45] [color={rgb, 255:red, 0; green, 0; blue, 0 }  ][fill={rgb, 255:red, 0; green, 0; blue, 0 }  ][line width=0.75]      (0, 0) circle [x radius= 1.34, y radius= 1.34]   ;
        \draw [shift={(100.25,179.87)}, rotate = 90.45] [color={rgb, 255:red, 0; green, 0; blue, 0 }  ][fill={rgb, 255:red, 0; green, 0; blue, 0 }  ][line width=0.75]      (0, 0) circle [x radius= 1.34, y radius= 1.34]   ;
        \draw    (140.25,205.87) -- (99.85,230.27) ;
        \draw [shift={(99.85,230.27)}, rotate = 148.87] [color={rgb, 255:red, 0; green, 0; blue, 0 }  ][fill={rgb, 255:red, 0; green, 0; blue, 0 }  ][line width=0.75]      (0, 0) circle [x radius= 1.34, y radius= 1.34]   ;
        \draw [shift={(140.25,205.87)}, rotate = 148.87] [color={rgb, 255:red, 0; green, 0; blue, 0 }  ][fill={rgb, 255:red, 0; green, 0; blue, 0 }  ][line width=0.75]      (0, 0) circle [x radius= 1.34, y radius= 1.34]   ;
        \draw    (140.25,205.87) -- (100.25,179.87) ;
        \draw [shift={(100.25,179.87)}, rotate = 213.02] [color={rgb, 255:red, 0; green, 0; blue, 0 }  ][fill={rgb, 255:red, 0; green, 0; blue, 0 }  ][line width=0.75]      (0, 0) circle [x radius= 1.34, y radius= 1.34]   ;
        \draw [shift={(140.25,205.87)}, rotate = 213.02] [color={rgb, 255:red, 0; green, 0; blue, 0 }  ][fill={rgb, 255:red, 0; green, 0; blue, 0 }  ][line width=0.75]      (0, 0) circle [x radius= 1.34, y radius= 1.34]   ;
        \draw  [dash pattern={on 4.5pt off 4.5pt}]  (100.25,179.87) .. controls (147.85,173.01) and (242.25,182.21) .. (220.65,223.41) ;
        \draw [shift={(220.65,223.41)}, rotate = 117.67] [color={rgb, 255:red, 0; green, 0; blue, 0 }  ][fill={rgb, 255:red, 0; green, 0; blue, 0 }  ][line width=0.75]      (0, 0) circle [x radius= 1.34, y radius= 1.34]   ;
        \draw [shift={(100.25,179.87)}, rotate = 351.79] [color={rgb, 255:red, 0; green, 0; blue, 0 }  ][fill={rgb, 255:red, 0; green, 0; blue, 0 }  ][line width=0.75]      (0, 0) circle [x radius= 1.34, y radius= 1.34]   ;
        \draw  [dash pattern={on 4.5pt off 4.5pt}]  (99.85,230.27) .. controls (147.45,223.41) and (197.45,257.41) .. (220.65,223.41) ;
        \draw [shift={(220.65,223.41)}, rotate = 304.31] [color={rgb, 255:red, 0; green, 0; blue, 0 }  ][fill={rgb, 255:red, 0; green, 0; blue, 0 }  ][line width=0.75]      (0, 0) circle [x radius= 1.34, y radius= 1.34]   ;
        \draw [shift={(99.85,230.27)}, rotate = 351.79] [color={rgb, 255:red, 0; green, 0; blue, 0 }  ][fill={rgb, 255:red, 0; green, 0; blue, 0 }  ][line width=0.75]      (0, 0) circle [x radius= 1.34, y radius= 1.34]   ;
        \draw  [dash pattern={on 4.5pt off 4.5pt}]  (140.25,205.87) .. controls (187.85,199.01) and (195.85,208.21) .. (220.65,223.41) ;
        \draw [shift={(220.65,223.41)}, rotate = 31.5] [color={rgb, 255:red, 0; green, 0; blue, 0 }  ][fill={rgb, 255:red, 0; green, 0; blue, 0 }  ][line width=0.75]      (0, 0) circle [x radius= 1.34, y radius= 1.34]   ;
        \draw [shift={(140.25,205.87)}, rotate = 351.79] [color={rgb, 255:red, 0; green, 0; blue, 0 }  ][fill={rgb, 255:red, 0; green, 0; blue, 0 }  ][line width=0.75]      (0, 0) circle [x radius= 1.34, y radius= 1.34]   ;
        
        \draw (87.2,175.81) node [anchor=north west][inner sep=0.75pt]  [font=\footnotesize]  {$v_{1}$};
        \draw (85.6,215.81) node [anchor=north west][inner sep=0.75pt]  [font=\footnotesize]  {$v_{2}$};
        \draw (134.8,189.41) node [anchor=north west][inner sep=0.75pt]  [font=\footnotesize]  {$v_{3}$};
        \draw (176.4,207.81) node [anchor=north west][inner sep=0.75pt]  [font=\footnotesize]  {$P'$};
        \draw (196.4,170.61) node [anchor=north west][inner sep=0.75pt]  [font=\footnotesize]  {$P$};
        \draw (37.6,193.41) node [anchor=north west][inner sep=0.75pt]    {$H=$};
        \end{tikzpicture}
    \end{figure}
        Since $G$ has no clique cutset, $G\backslash K$ is connected, so let $P$ be a shortest path in $G\backslash K$ from $u_1$ to $u_2$. Let $P'$ be a shortest path among all those in $G\backslash K$ whose endpoints are $u_3$ and an $x\in P$. Now let $H$ be the subgraph induced on $\{v_1,v_2,v_3\} \cup V(P) \cup V(P')$; we'll show it has the desired properties. We know that $H\backslash \{v_1,v_2,v_3\}$ is connected, so we can contract all of $H\backslash \{v_1,v_2,v_3\}$. This yields $K_4$, so $H$ has a $K_4$ minor. 
    
        By $N(v_1)\cap N(v_2)\cap V(G\backslash K)=\emptyset$ and the definition of $P$, no $K_4\subseteq H$ can contain more than two elements of $\{v_1,v_2\}\cup V(P)$, and they can't be $v_1,v_2$. Thus, any $K_4$ must include some vertex of $V(P)$ and at least two vertices of $V(P')\backslash V(P)$. However, by the definition of $P'$, at most one vertex of $V(P')\backslash V(P)$ can have neighbors in $V(P)$. Thus, $H$ contains no $K_4$ subgraph. 
        
        \emph{Case 3.} Assume that $|V(K)|\geq 3$ but every pair of vertices in $K$ has some common neighbor in $G\backslash K$. Of course, not all of $K$ can have a common neighbor in $G\backslash K$; otherwise it wouldn't be a maximum clique. Thus, there must exist some $I\subseteq V(K)$ with $|I|\geq 3$ having the property that $I$ has no common neighbor in $G\backslash K$ but every proper subset of $I$ does. 
        
        Let $A,B\subset I$ be proper subsets of some such $I$, and pick any $a\in A\backslash B$, $b\in B\backslash A$, $c\in A\cap B$. Since $G$ is has no cut-clique,$G\backslash K$ is connected, so there there is a minimal path $P_{b,v}$ in $(G\backslash K) \cup \{b\}$ from $b$ to any given vertex $v$ outside $K$. Choose $d\in N(a)\cap N(c)\cap V(G\backslash K)$ to minimize $|V(P_{b,d})|$. Now consider the graph induced on $\{a,c\}\cup V(P_{b,d})$; that will be our $H$.

\includegraphics{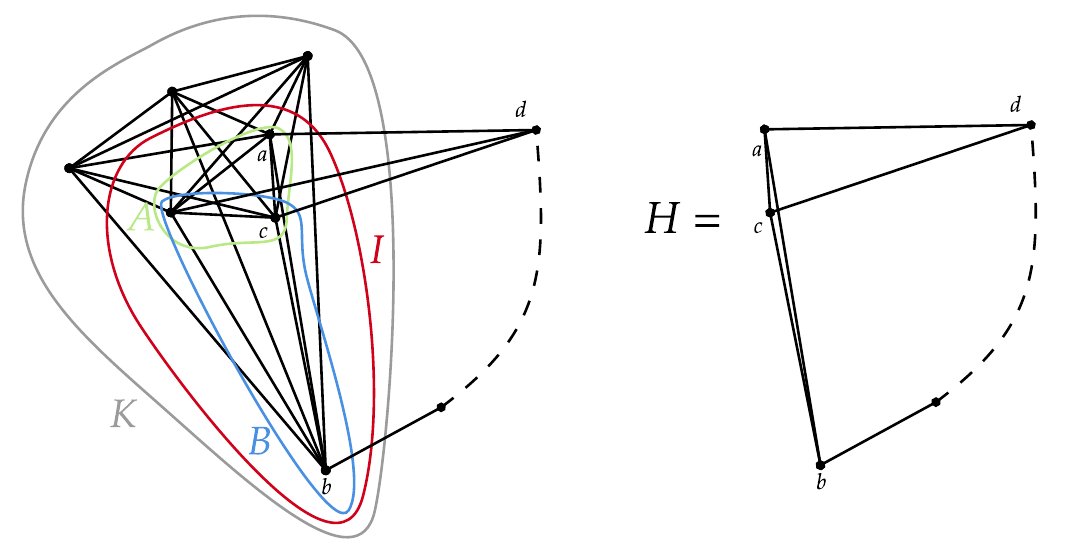}

        Contracting $P_{b,d}$ to a single edge gives us $K_4$, so $H$ has a $K_4$ minor.
        No vertex of $P_{b,d}\backslash \{b,d\}$ can neighbor both $a$ and $c$, since we chose $d$ to minimize $|P_{b,d}|$, and every $v_i\in V(P_{b,d})\backslash \{b,d\}$ has exactly two neighbors $v_{i-1},v_{i+1}$ in $V(P_{b,d})$ by minimality of $P_{b,d}$. By minimality of $P_{b,d}$, $v_{i-1}\not \sim v_{i+1}$. Thus, no $v_i\in V(P_{b,d})\backslash \{b,d\}$ can be part of a $K_4$ in $H$. There are only four other vertices ($a,b,c,d$)  of $H$, but $b\not\sim d$. Thus, $H$ contains no $K_4$ subgraph.
        
        \emph{To conclude.} In each of these cases, there's an induced $H\subseteq G$ that has a $K_4$ minor but contains no $K_4$ subgraph. 
    \end{proof}

\begin{lmm}\label{WheelsareF}
    If $H$ has a $G$ minor, so does everything built from $H$. If every edge of $G$ is in a triangle and $H$ contains no $G$ subgraph, nothing built from $H$ contains a $G$ subgraph. 
\end{lmm}

\begin{proof}
    Let $H'$ be obtained from $H$ by a single vertex partition, splitting vertex $v$ into vertices $x$ and $y$. 
    By contracting $xy$, we see that $H$ is a minor of $H'$, so all minors of $H$ are minors of $H'$. This proves the lemma's first sentence. 
    
    Consider an induced subgraph $G$ of $H'$. If it includes neither of $x,y$, it is an induced subgraph in  $H$. If it includes exactly one of $x,y$, then $(G\backslash \{x, y\}) \cup \{v\}$ is an induced $G\subseteq H$ of the same size. It cannot include both of $x,y$ because $xy$ is part of no triangle by the definition of vertex partition, but $G$ has no such edge. This contradiction proves the second sentence. 
\end{proof}

\begin{lmm}\label{Mostboringlemma}
    If $G$ is in $\mathcal{F}$, then $G$ has treewidth $3$ and no cut-clique. 
\end{lmm}

\begin{proof}
    Inspection. 
\end{proof}

\begin{lmm}\label{3kimpliesAk}
    If $G$ has a $(k+1)$-clique chordal supergraph, then it is a clique-sum of cliques and graphs of treewith $\leq k$.
\end{lmm}

\begin{proof}
    Let $H$ be a $(k+1)$-clique chordal supergraph of $G$. It has a tree decomposition $T$ each of whose bags is a clique. Moreover, a bag has $>k+1$ members if and only if its clique is also a clique in $G$. Since $G\subseteq H$, the vertex assignments of $T$ can be duplicated for a tree decomposition $T'$ of $G$; each bag has $\leq k+1$ members or is a clique of $G$. 
    
    If all bags of $T'$ have $\leq k+1$ members, then by definition $G$ has treewidth $\leq k$. 
    
    Otherwise, let $B_0$ a clique of size $>k+1$ that is a bag of $T'$. If $B_0$ is the only bag, we're done. Otherwise, it has a neighbor $B_1$. If $B_0\neq B_0 \cap B_1\neq B_1$, then $B_0\cap B_1$ is a cut-clique of $G$. If $B_0\subseteq B_1$ or $B_1\subseteq B_0$, we can merge the bags. By induction, $G$ has a cut-clique or is a clique.
    
    We have shown that if $G$ has a $(k+1)$-clique chordal supergraph, then at least one of the following holds:
    \begin{itemize}
        \item $G$ has treewidth $\leq k$.
        \item $G$ is a clique.
        \item $G$ has a clique cutset.
    \end{itemize}
    In the first two cases, $G$ is trivially an appropriate clique-sum. These (or $K_1$) can serve as the base case for the induction we use to deal with the third possibility. Our inductive hypothesis is that the theorem holds for all graphs with fewer than $|V(G)|$ vertices. 
    
    If $G$ has a $(k+1)$-clique chordal supergraph, so does every induced subgraph. Let $K$ be a clique cutset of $G$, and consider the induced subgraphs $G_i$ on vertex-sets of the form $V(K)\cup V(C_i)$ for $C_i$ the components of $G\backslash K$. By the inductive hypothesis, each of $G_i$ is the clique-sum of cliques and series-parallel graphs. In turn, $G$ is the clique-sum of those $G_i$, so $G$ is the clique-sum of cliques and series-parallel graphs, and the proof is complete.
    \end{proof}
        
\begin{proof}[\rm \bf Proof of Theorem \ref{originalequivalence}]
    The paragraphs below will leverage a lemma from \cite{LMT} and our own Lemmata 3.1-3.4
    to help prove $(\neg 2) \rightarrow (\neg \alpha), (\neg \alpha) \rightarrow (\neg 1)$, $(\neg A) \rightarrow (\neg 2)$, $(\neg 1) \rightarrow (\neg 2)$, and $(\neg \alpha) \rightarrow (\neg A)$. 
    This gives us $(1)\leftrightarrow (2) \leftrightarrow (\alpha) \leftrightarrow (A)$. As for $(B)$ and $(3)$, the original proof of $(A)\rightarrow (B) \rightarrow (3)$ is too good to change, so we just append $(3)\rightarrow (A)$, which is a special case of Lemma \ref{3kimpliesAk}. Each paragraph below is labeled with the implication proven therein.

    $(\neg 2) \rightarrow (\neg \alpha)$. To rephrase Lemma 2.2 of \cite{LMT}, if a graph contains no induced subgraph equal to or built from $K_4$ but still has a $K_4$ minor, it contains a prism, wheel, or $K_{3,3}$ as an induced subgraph. So if some induced $H\subseteq G$ has a $K_4$ minor but contains no $K_4$ subgraph, it is not $\mathcal{F}$-free, which implies that $G$ itself is not $\mathcal{F}$-free. 

    $(\neg \alpha) \rightarrow (\neg 1)$. Assume $G$ contains an induced $F\subseteq G$ with $F\in \mathcal{F}$. If that $F$ is equal to or built from $\widehat{W}_4$, we have $(\neg 1)$. Otherwise, that $F$ is a prism, wheel, or $K_{3,3}$, and Lemma \ref{Fiswheelfounded} yields $(\neg 1)$.

    $(\neg A) \rightarrow (\neg 2)$. If every induced subgraph of $G$ is series-parallel, is a clique, or has a cut-clique, then $G$ is the clique-sum of cliques and series-parallel graphs. So if $G$ is not the clique-sum of cliques and series-parallel graphs, it contains some induced subgraph $G'$ failing all three of those conditions. By Lemma \ref{Onlyhardlemma}, this $G'$ contains an induced subgraph $H$ that is not series-parallel nor contains a $K_4$ subgraph, so $G$ does as well. 

    $(\neg 1) \rightarrow (\neg 2)$. Check that wheels have $K_4$ minors and contain no $K_4$ subgraphs, then use Lemma \ref{WheelsareF}.
    
    $(\neg \alpha) \rightarrow (\neg A)$. Assume for contradiction that $G$ is a clique-sum of cliques and series-parallel graphs, minimal among those that are not $\mathcal{F}$-free. It cannot be a clique, since no induced subgraph of a clique lies in $\mathcal{F}$. By Lemma \ref{Mostboringlemma} and the monotonicity of treewidth, it cannot be series-parallel. 
    Assume it has a clique cutset $K$. 
    Therefore, it must have a clique cutset $K$. Let $\{C_i\}$ be the components of $G\backslash K$. Let $H\in \mathcal{F}$ be an induced subgraph of $G$. By the minimality assumption, there is no $G|_{V(C_i)\cup K}$  with $H\subseteq G|_{V(C_i)\cup K}$. Thus, multiple $C_i$ must contain vertices of $H$. Since $H$ is connected, $K\cap H$ must be a cutset of $H$, and since $H$ is an induced subgraph of $G$, $K\cap H$ is a clique in $H$. By Lemma \ref{Mostboringlemma}, this contradiction concludes the proof. 
    
    $(3)\rightarrow (A)$. Recall that series-parallel graphs are exactly the graphs of treewidth $\leq 2$. Then consider the case $k=3$ of Lemma \ref{3kimpliesAk}.
\end{proof}


\section{Generalizing}\label{Generalizing}
Here we prove Theorem \ref{TheKversion}. Notice that Lemma \ref{3kimpliesAk} is $(3_k)\rightarrow (A_k)$. We could prove $(A_k)\rightarrow (B_k)\rightarrow (3_k)$ by a straightforward generalization of \cite{JM}'s proof that $(A)\rightarrow (B) \rightarrow (3)$, but for novelty's sake, we follow another pattern. 
\begin{lmm}\label{3kimpliesBk}
    Let $G'$ be a $(k+1)$-clique chordal supergraph of $G$, and let $\pi$ be any perfect elimination ordering for $G'$. Then $\pi$ is a $k$-blended elimination ordering for $G$. 
\end{lmm}
\begin{lmm}\label{Bkimplies3k}
    Let $\pi$ be a $k$-blended elimination ordering for $G$. Then $G_\pi$ is a $(k+1)$-clique chordal supergraph of $G$.  
\end{lmm}
\begin{lmm}\label{Akimplies3k}
    A clique-sum of cliques and graphs with treewidth $\leq k$ has a $(k+1)$-clique chordal supergraph.
\end{lmm}
\begin{proof}[\rm \bf Proof of Theorem \ref{TheKversion}]
    Lemma \ref{3kimpliesAk} and Lemma \ref{Akimplies3k} together constitute $(A_k)\leftrightarrow (3_k)$. Lemma 4.1 implies that if a graph has a $(k+1)$-clique chordal supergraph, then it has a $k$-blended elimination ordering. Lemma 4.2 implies the converse. Together, they imply $(3_k)\leftrightarrow (B_k)$. 
\end{proof}

\begin{proof}[Proof of Lemma \ref{3kimpliesBk}]
    Notice that $G_\pi$ is a subgraph of $G'$. Then our conclusion follows from the definitions of $k$-blended elimination ordering and $(k+1)$-clique supergraph. 
\end{proof}

\begin{proof}[Proof of Lemma \ref{Bkimplies3k}] See definition of $k$-blended elimination ordering. 
\end{proof}

\begin{proof}[Proof of Lemma \ref{Akimplies3k}]
    Consider the cut-clique decomposition of some $G$ into cliques and partial $k$-trees. The latter have chordal supergraphs with maximum clique size $k+1$. Take these chordal supergraphs. Reassembling $G$ from this cut-clique decomposition only requires the identification of existing cliques with each other; this is a $(k+1)$-clique chordal supergraph of $G$.
\end{proof}

\section{Acknowledgements}
The first author is partially supported by NSF-EPSRC Grant DMS-2120644 and by AFOSR grant FA9550-22-1-0083. Both authors thank the generous, meticulous reviewers.

\end{document}